\documentclass{amsart}

\usepackage{graphicx}
\usepackage[centertags]{amsmath}
\usepackage{amsfonts}
\usepackage{amssymb}
\usepackage{amsthm}
\usepackage{newlfont}
\usepackage{hyperref}
\usepackage{mathrsfs} %math text caligraphy
\usepackage{yfonts} %math text goth

\newtheorem*{thm*}{Theorem}
\newtheorem{thm}{Theorem}[section]
\newtheorem{cor}[thm]{Corollary}
\newtheorem{lem}[thm]{Lemma}

\theoremstyle{definition}
\newtheorem{defn}[thm]{Definition}

\theoremstyle{remark}

\numberwithin{equation}{section}

\DeclareSymbolFont{bbold}{U}{bbold}{m}{n}
\DeclareSymbolFontAlphabet{\mathbbold}{bbold}

\newcommand{\N}{\mathbb{N}}
\newcommand{\Z}{\mathbb{Z}}
\newcommand{\R}{\mathbb{R}}

\newcommand{\C}{\mathbb{C}}

\newcommand{\acts}{\curvearrowright}
\newcommand{\Stab}{\mathrm{Stab}}

\newcommand{\pmp}{p{$.$}m{$.$}p{$.$}}

\newcommand{\cA}{\mathcal{A}}
\newcommand{\cF}{\mathcal{F}}

\newcommand{\sP}{\mathscr{P}}

\newcommand{\symd}{\triangle}
\newcommand{\salg}{\sigma \text{-}\mathrm{alg}}

\newcommand{\sH}{\mathrm{H}}
\newcommand{\rh}{h}
\newcommand{\ksh}{h^{\mathrm{KS}}}

\newcommand{\Borel}{\mathcal{B}}
\newcommand{\given}{\mathbin{|}}

\newcommand{\res}{\restriction}

\newcommand{\BL}{\mathscr{L}}
\newcommand{\HS}{\mathscr{H}}
\newcommand{\cU}{\mathcal{U}} % unitary operators
\newcommand{\cB}{B} % bounded operators
\newcommand{\Exp}{\mathbb{E}}

\newcommand{\aprod}{\rtimes_{\mathrm{alg}}}
\newcommand{\Sub}{\mathrm{Sub}}

\newcommand{\sinai}{Sinai}

\begin{document}

\title[The Koopman representation and positive Rokhlin entropy]{The Koopman representation and positive Rokhlin entropy}
\author{Brandon Seward}
\address{Courant Institute of Mathematical Sciences, New York University, 251 Mercer Street, New York, NY 10003, U.S.A.}
\email{b.m.seward@gmail.com}
%\subjclass{37A35} %Entropy and other invariants, isomorphism, classification
%\subjclass{37A15} %General groups of measure-preserving transformations

\begin{abstract}
In a prior paper, the author generalized the classical factor theorem of {\sinai} to actions of arbitrary countably infinite groups. In the present paper, we use this theorem and the techniques of its proof in order to study connections between positive entropy phenomena and the Koopman representation. Following the line of work initiated by Hayes for sofic entropy \cite{H, Ha}, we show in a certain precise manner that all positive entropy must come from portions of the Koopman representation that embed into the left-regular representation. By combining this result with the generalized factor theorem of the previous paper, we conclude that for actions having completely positive outer entropy, the Koopman representation must be isomorphic to the countable direct sum of the left-regular representation. This generalizes a theorem of Dooley--Golodets for countable amenable groups. As a final consequence, we observe that actions with completely positive outer entropy must be mixing, and when the group is non-amenable they must be strongly ergodic and have spectral gap. We also prove generalizations of these results that apply relative to sub-$\sigma$-algebras and apply to non-free actions.
\end{abstract}
\maketitle

\section{Introduction}

In this paper we use entropy theory to study the spectral structure of probability-measure-preserving ({\pmp}) actions of general countable groups. We will use an extension of the classical Kolmogorov--{\sinai} entropy called Rokhlin entropy that was introduced by the author in 2014 \cite{S1} and is defined for actions of general countable groups. If $G$ is a countable group, $G \acts (X, \mu)$ is an aperiodic {\pmp} action, $\cF$ is a $G$-invariant sub-$\sigma$-algebra, and $\xi \subseteq \Borel(X)$ is a collection of sets, then the \emph{(outer) Rokhlin entropy of $\xi$ relative to $\cF$} is defined to be
$$\rh_G(\xi \given \cF) = \inf \Big\{ \sH(\alpha \given \cF) : \alpha \text{ a countable partition with } \xi \subseteq \salg_G(\alpha) \vee \cF \Big\},$$
where $\salg_G(\alpha)$ is the smallest $G$-invariant $\sigma$-algebra containing $\alpha$, and $\sH(\alpha \given \cF)$ is the conditional Shannon entropy of $\alpha$ relative to $\cF$. When $\xi = \Borel(X)$ is the Borel $\sigma$-algebra of $X$, we write instead $\rh_G(X, \mu \given \cF)$ for the \emph{Rokhlin entropy of $(X, \mu)$ relative to $\cF$}. When $\cF = \{X, \varnothing\}$ is trivial we simply write $\rh_G(\xi)$ and $\rh_G(X, \mu)$ for the Rokhlin entropy of $\xi$ and the Rokhlin entropy of $(X, \mu)$, respectively. In this paper, $\rh_G$ will always denote Rokhlin entropy.

A closely related concept is the notion of sofic entropy. Sofic entropy is an entropy notion for {\pmp} actions of countable sofic groups that was created in 2008 through ground-breaking work of Lewis Bowen \cite{B10b} and improvements by Kerr and Li \cite{KL11a}. Both sofic entropy and Rokhlin entropy coincide with classical Kolmogorov--{\sinai} entropy when the acting group is amenable \cite{AS, B12, KL13}, and it is an important open question whether Rokhlin entropy is equal to sofic entropy whenever sofic entropy is defined and not minus infinity.

The present paper is a sequel to \cite{S18}. In \cite{S18} the author proved the following theorem, which is a generalization of the classical {\sinai} factor theorem for $\Z$.

\begin{thm}[Seward, \cite{S18}] \label{intro:sinai}
Let $G$ be a countably infinite group, let $G \acts (X, \mu)$ be a free {\pmp} action, and let $(L, \lambda)$ be a standard probability space. If $\rh_G(X, \mu) \geq \sH(L, \lambda)$ then the action $G \acts (X, \mu)$ factors onto the Bernoulli shift $G \acts (L^G, \lambda^G)$.
\end{thm}

The arguments of the present paper rely heavily upon the above theorem, taken as a black box, and upon some of the ingredients appearing in its proof.

Spectral consequences of entropic phenomena were first uncovered by Rokhlin and {\sinai} in 1961 \cite{RS61}. They proved that the Koopman representation of any CPE action of the integers must be isomorphic to the countable direct sum of the left-regular representation (i.e. must have countable Lebesgue spectrum). Recall that an action $G \acts (X, \mu)$ of a countable amenable group $G$ is \emph{CPE}, or has \emph{completely positive entropy}, if every non-trivial factor action has strictly positive Kolmogorov--{\sinai} entropy. The Rokhlin--{\sinai} theorem was generalized to $\Z^d$ by Kami{\'n}ski in 1981 \cite{Ka81}, to $\bigoplus_{n \in \N} \Z$ by Kami{\'n}ski and Liardet in 1994 \cite{KaL94}, to finitely-generated torsion-free nilpotent groups by Golodets and Sinel'shchikov in 2002 \cite{GoSi02}, and finally to arbitrary countable amenable groups by Dooley and Golodets in 2002 \cite{DoGo}.

The definition of CPE immediately extends to the contexts of sofic and Rokhlin entropy, however there is a stronger alternative generalization of CPE in these settings as well. This alternative generalization arises from the fact that entropy is no longer monotone under factor maps. For example, for an action $G \acts (X, \mu)$ and a factor action $G \acts (Y, \nu)$ under a map $\phi : (X, \mu) \rightarrow (Y, \nu)$, it is possible that $\rh_G(Y, \nu) > 0$ while $\rh_G(\phi^{-1}(\Borel(Y))) = 0$. Thus, we can say an action $G \acts (X, \mu)$ has \emph{completely positive outer Rokhlin entropy}, written CPE$^+$, if $\rh_G(\cF) > 0$ for all $\mu$-non-trivial $\cF \subseteq \Borel(X)$ (a similar strengthening of CPE exists in the setting of sofic entropy, but we won't need this). In this paper we will work with both CPE and CPE$^+$.

In the context of actions of non-amenable groups, the first exploration of the connection between entropy and the Koopman representation was undertaken by Ben Hayes using the framework of sofic entropy \cite{H,Ha}. His work, specifically the presentation of his results, strongly influenced the nature of the present paper. We emphasize that our proofs are entirely distinct from Hayes,' as his techniques are firmly rooted in finitary properties that are specific to soficity, while our techniques draw upon the ingredients of the proof of Theorem \ref{intro:sinai}.

Our main theorem is the following. It generalizes a nearly identical theorem due to Hayes in the case of sofic entropy \cite{H}. Recall that two representations of $G$ are singular if no non-zero sub-representation of one embeds into the other.

\begin{thm} \label{intro:koop}
Let $G \acts (X, \mu)$ be a free {\pmp} action, and let $\rho : G \rightarrow \cU(\BL^2_\mu(X))$ be the corresponding Koopman representation. Also let $\lambda : G \rightarrow \cU(\ell^2(G))$ be the left-regular representation. If $\HS$ is a $\rho(G)$-invariant closed subspace of $\BL^2_\mu(X)$ and $\rho|_\HS$ is singular with $\lambda$, then $\rh_G(\salg(\HS)) = 0$.
\end{thm}

We mention that the theorem above and corollaries below are stated only in their simplest form, as we in fact prove both relative and non-free versions of these results (similar to the relative results of Hayes \cite{Ha}). See Theorem \ref{thm:koop} and Corollaries \ref{cor:cpekoop} and \ref{cor:mix}.

With the above theorem we are able to compute the Rokhlin entropy of some Gaussian actions.

\begin{cor}
Let $G$ be a countably infinite group, let $\pi : G \rightarrow \mathcal{O}(\HS)$ be an orthogonal representation on a real separable Hilbert space $\HS$, and let $\lambda_\R : G \rightarrow \mathcal{O}(\ell^2(G, \R))$ be the real left-regular representation of $G$. Suppose that $\pi$ is singular with $\lambda_\R$. Then the Gaussian action $G \acts (X_\pi, \mu_\pi)$ induced by $\pi$ satisfies $\rh_G(X_\pi, \mu_\pi) = 0$.
\end{cor}

The sofic entropy of Gaussian actions of sofic groups was previously computed by Hayes \cite{Hb} in the case where the sofic entropy is not minus infinity. Combined with the above corollary, we immediately obtain the following consequence.

\begin{cor}
Let $G$ be a countably infinite sofic group and let $\pi : G \rightarrow \mathcal{O}(\HS)$ be an orthogonal representation on a real separable Hilbert space $\HS$. Then the Rokhlin entropy of the corresponding Gaussian action $G \acts (X_\pi, \mu_\pi)$ is equal to the sofic entropy whenever the sofic entropy is not minus infinity.
\end{cor}

By applying a simple Zorn's lemma argument to Theorem \ref{intro:koop}, one can conclude that for a free CPE$^+$ {\pmp} action $G \acts (X, \mu)$, the Koopman representation restricted to the orthogonal complement of the constants, $\rho|_{\BL^2_\mu(X) \ominus \C}$, must embed into $\lambda^{\oplus \N}$. Conversely, since the Koopman representation of any Bernoulli shift is isomorphic to $\lambda^{\oplus \N}$, it follows from Theorem \ref{intro:sinai} that $\lambda^{\oplus \N}$ embeds into the Koopman representation of any free ergodic {\pmp} action with positive Rokhlin entropy. This leads to the following.

\begin{cor} \label{intro:cpe}
Let $G \acts (X, \mu)$ be a free {\pmp} action, and let $\rho$ be the corresponding Koopman representation. Also let $\lambda$ denote the left-regular representation of $G$. Assume that either (1) $G \acts (X, \mu)$ is CPE$^+$, or (2) $G \acts (X, \mu)$ is CPE and every non-trivial factor action is essentially free. Then $\rho|_{\BL^2_\mu(X) \ominus \C}$ is isomorphic to $\lambda^{\oplus \N}$.
\end{cor}

Similarly, as formally recorded in \cite{S18}, Theorem \ref{intro:sinai} combines with the result of Hayes \cite{H} to provide a sofic entropy version of the above corollary.

Finally, we observe that Corollary \ref{intro:cpe} implies the following.

\begin{cor}
Let $G \acts (X, \mu)$ be a free {\pmp} action such that either (1) $G \acts (X, \mu)$ is CPE$^+$, or (2) $G \acts (X, \mu)$ is CPE and every non-trivial factor action is essentially free. Then $G \acts (X, \mu)$ is mixing and for every non-amenable $\Gamma \leq G$ the restricted action $\Gamma \acts (X, \mu)$ has spectral gap and is strongly ergodic.
\end{cor}

\subsection*{Acknowledgments}
This research was partially supported by ERC grant 306494 and Simons Foundation grant 328027 (P.I. Tim Austin). The author thanks Tim Austin and Ben Hayes for helpful and encouraging conversations.

\section{Generalized *-algebras and *-representations}

For a countable set $A$ we write $\ell^2(A)$ for the Hilbert space of square-summable complex-valued functions on $A$, and for $a \in A$ we write $\delta_a$ for the function that is $1$ at $a$ and $0$ elsewhere. If $\HS$ is a Hilbert space then we write $\cB(\HS)$ and $\cU(\HS)$ for the set of bounded operators and the set of unitary operators, respectively. Recall that the \emph{left-regular representation} of $G$ is the unitary representation $\lambda : G \rightarrow \cU(\ell^2(G))$ defined by $\lambda_t \delta_g = \delta_{t g}$. Also recall that for a {\pmp} action $G \acts (X, \mu)$, the \emph{Koopman representation} of $G$ associated to this action is the unitary representation $\rho : G \rightarrow \cU(\BL^2_\mu(X))$ defined by $(\rho(g) f)(x) = f(g^{-1} \cdot x)$.

In the simplest context of (standard non-relative) entropy and free actions, the results of this paper use entropy to compare the Koopman representation $\rho$ and the left-regular representation $\lambda$. However we will work in slightly greater generality by considering both relative entropy and actions which are not necessarily free. In order to do this we will need to work in the setting of $*$-representations of $*$-algebras rather than unitary representations of $G$. This more general perspective was initiated by Ben Hayes \cite{Ha}, and our treatment closely follows his.

In this section we introduce a class of $*$-algebras and corresponding generalizations of the Koopman and left-regular representations. Let us first introduce the $*$-algebras we will work with. Let $G \acts (Y, \nu)$ be a {\pmp} action. Define the algebraic crossed product $\BL^\infty_\nu(Y) \aprod G$ to be the algebra of all finite formal sums
$$\sum_{g \in G} s_g u_g, \quad s_g \in \BL^\infty_\nu(Y) \text{ and } s_g = 0 \text{ for all but finitely many } g,$$
with the relation $u_g s = (s \circ g^{-1}) u_g$. This algebra becomes a $*$-algebra with the convention that
$$\Big( \sum_{g \in G} s_g u_g \Big)^* = \sum_{g \in G} ( \bar{s}_g \circ g) u_{g^{-1}}.$$

Next we introduce a generalized Koopman representation for $\BL^\infty_\nu(Y) \aprod G$. Let $G \acts (X, \mu)$ be a {\pmp} action and let $\phi : (X, \mu) \rightarrow (Y, \nu)$ be a $G$-equivariant factor map. We naturally obtain a $*$-representation $\rho : \BL^\infty_\nu(Y) \aprod G \rightarrow \cB(\BL^2_\mu(X))$ defined by
$$(\rho(s u_g) f)(x) = s(\phi(x)) f(g^{-1} \cdot x).$$
Using $\phi$ we may view $\BL^2_\nu(Y)$ as a subspace of $\BL^2_\mu(X)$. Often we will focus on the restriction of $\rho$ to $\BL^2_\mu(X) \ominus \BL^2_\nu(Y)$.

Lastly, we we introduce a generalized left-regular representation. Denote by $\Sub(G)$ the topological space of all subgroups of $G$. A base for the topology on $\Sub(G)$ is given by the basic open sets $\{H \in \Sub(G) : H \cap T = F\}$ as $F \subseteq T$ range over the finite subsets of $G$. We let $G \acts \Sub(G)$ by conjugation. Fix a $G$-invariant Borel probability measure $\omega$ on $\Sub(G) \times Y$ such that the projection to $Y$ pushes $\omega$ forward to $\nu$. We define an associated $*$-representation
$$\lambda : \BL^\infty_\nu(Y) \aprod G \rightarrow \cB \left( \int_{\Sub(G) \times Y}^\oplus \ell^2(\Gamma \backslash G) \ d \omega(\Gamma, y) \right)$$
by
$$(\lambda(s u_g) f)(\Gamma, y)(\Gamma h) = s(y) f(g^{-1} \Gamma g, g^{-1} \cdot y)(g^{-1} \Gamma h).$$
In particular, if $\xi$ is defined by $\xi(\Gamma, y) = \delta_{\Gamma}$ then
$$(\lambda(s u_g) \xi)(\Gamma, y) = s(y) \delta_{\Gamma g}.$$
For later reference, notice that there is no proper closed $\BL^\infty_\nu(Y) \aprod G$-invariant subspace containing $\xi$, and note that for $s, t \in \BL^\infty_\nu(Y)$ and $g, h \in G$ we have
\begin{align}
\langle \lambda(s u_g) \xi, \lambda(t u_h) \xi \rangle & = \int_{\Sub(G) \times Y} s(y) \overline{t(y)} \langle \delta_{\Gamma g}, \delta_{\Gamma h} \rangle \ d \omega(\Gamma, y)\nonumber\\
 & = \int_{S_{g h^{-1}} \times Y} s(y) \overline{t(y)} \ d \omega(\Gamma, y),\label{eqn:leftreg}
\end{align}
where $S_{g h^{-1}} = \{\Gamma \in \Sub(G) : g h^{-1} \in \Gamma\}$.

\begin{lem} \label{lem:repiso}
Let $G \acts (Y, \nu)$ be a {\pmp} action, and let $\pi_i : \BL^\infty_\nu(Y) \aprod G \rightarrow \cB(\HS_i)$ be two $*$-representations. If $\pi_1$ and $\pi_2$ each embed into one another then they are isomorphic.
\end{lem}

\begin{proof}
Set $\pi = \pi_1 \oplus \pi_2$ and $\HS = \HS_1 \oplus \HS_2$. Consider the von Neumann algebra $M = \pi(\BL^\infty_\nu(Y) \aprod G)'$. Let $V : \HS_1 \rightarrow \HS_2$ be a bounded $\BL^\infty_\nu(Y) \aprod G$-equivariant linear isometry, and define the operator $\tilde{V}$ on $\HS$ by $\tilde{V}(\xi, \eta) = (0, V(\xi))$.  Also let $P_i : \HS \rightarrow \HS_i$, $i = 1, 2$, be the orthogonal projections. Note that $\tilde{V}, P_1, P_2 \in M$. A simple computation shows that $\tilde{V}^* \tilde{V} = P_1$ and $\tilde{V} \tilde{V}^* \leq P_2$, so $P_1$ is sub-equivalent to $P_2$. By symmetry, we also have $P_2$ is sub-equivalent to $P_1$. By Murray-von Neumann equivalence of projections in von Neumann algebras \cite[Prop. V.1.3]{Ta02}, there is $W \in M$ with $W^* W = P_1$ and $W W^* = P_2$. Note that $W \in M$ implies $W$ is $\pi(\BL^\infty_\nu(Y) \aprod G)$-equivariant. One can then check that the restriction of $W$ to $\HS_1 \oplus \{0\}$ is a $\BL^\infty_\nu(Y) \aprod G$-equivariant isometry onto $\{0\} \oplus \HS_2$.
\end{proof}

\section{The Koopman representation of generalized Bernoulli shifts}

It is well known that for any non-trivial probability space $(K, \kappa)$ and any countable group $G$, the Koopman representation of $G$ associated with the Bernoulli shift action $G \acts (K^G, \kappa^G)$ is, after restricting to the orthogonal complement of the constants, isomorphic to $\lambda^{\oplus \N}$, the countable direct-sum of the left-regular representation (i.e. there is a unitary $G$-equivariant isomorphism from $\BL^2_{\kappa^G}(K^G) \ominus \C$ to $\ell^2(G)^{\oplus \N}$). In this section we aim to simply prove containment of $\lambda^{\oplus \N}$ in a more general framework.

Let $(K, \kappa)$ be a standard probability space. We let $G$ act on $K^G$ by the standard left-shift action: $(g \cdot x)(t) = x(g^{-1} t)$ for $g, t \in G$ and $x \in K^G$. For $\Gamma \in \Sub(G)$, we identify $K^{\Gamma \backslash G}$ with the set of points $x \in K^G$ with $\Gamma \subseteq \Stab(x)$, and we consider the corresponding Borel probability measure $\kappa^{\Gamma \backslash G}$ on $K^G$ which is supported on $K^{\Gamma \backslash G}$. If we pick any conjugation-invariant Borel probability measure $\theta$ on $\Sub(G)$ (this is known as an IRS \cite{AGV}), then we obtain a non-free Bernoulli measure on $K^G$ by integrating:
$$\kappa^{\theta \backslash G} := \int_{\Sub(G)} \kappa^{\Gamma \backslash G} \ d \theta(\Gamma).$$
If $\Gamma \in \Sub(G)$ has infinite index in $G$ and $\kappa$ is non-trivial, then $\Stab(x) = \Gamma$ for $\kappa^{\Gamma \backslash G}$-almost-every $x \in K^G$, where $\Stab(x) = \{g \in G : g \cdot x = x\}$ is the stabilizer map. Thus $\theta = \Stab_*(\kappa^{\theta \backslash G})$ whenever $\theta$ is supported on the infinite-index subgroups of $G$. Note that if $\theta = \Stab_*(\mu)$ for a {\pmp} action $G \acts (X, \mu)$, then $\theta$ is supported on the infinite-index subgroups of $G$ if and only if almost-every point has an infinite orbit (such an action is said to be \emph{aperiodic}).

\begin{lem} \label{lem:leftreg}
Let $G \acts (Y, \nu)$ be a {\pmp} action and let $y \in Y \mapsto \Gamma_y \in \Sub(G)$ be a $G$-equivariant Borel map where each $\Gamma_y$ is of infinite index. Let
$$\lambda : \BL^\infty_\nu(Y) \aprod G \rightarrow \cB \left( \int_Y^\oplus \ell^2(\Gamma_y \backslash G) \ d \nu(y) \right)$$
be the generalized left-regular $*$-representation. Fix a non-trivial probability space $(K, \kappa)$ and consider the action $G \acts (X, \mu)$ where
$$(X, \mu) = \left(K^G \times Y, \int_Y \kappa^{\Gamma_y \backslash G} \times \delta_y \ d \nu(y) \right).$$
Let $\rho : \BL^\infty_\nu(Y) \aprod G \rightarrow \cB(\BL^2_\mu(X))$ be the generalized Koopman $*$-representation. Then $\lambda^{\oplus \N}$ embeds into $\rho|_{\BL^2_\mu(X) \ominus \BL^2_\nu(Y)}$.
\end{lem}

We believe that $\lambda^{\oplus \N}$ and $\rho|_{\BL^2_\mu(X) \ominus \BL^2_\nu(Y)}$ are isomorphic, though we have not checked this carefully. The above statement is all that we will need for our theorems.

\begin{proof}
Let $S$ denote the set of subgroups of $G$ having infinite index. Our first goal is to define a sequence of Borel functions $\Gamma \in S \mapsto R_\Gamma^n$, where $R_\Gamma^n \subseteq G$ is finite, satisfying the following condition for all $n, k \in \N$, $\Gamma \in S$, and $g \in G$:
\begin{enumerate}
\item [($*$)] if $g \not\in \Gamma$ or $n \neq k$ then $\Gamma R_\Gamma^n \neq \Gamma g R_{g^{-1} \Gamma g}^k$.
\end{enumerate}
We partition $S$ into the following two conjugation invariant Borel sets
\begin{align*}
S_1 & = \{\Gamma \in S : |N_G(\Gamma) : \Gamma| = \infty\}\\
S_2 & = \{\Gamma \in S : |N_G(\Gamma) : \Gamma| < \infty\},
\end{align*}
where $N_G(\Gamma) = \{g \in G : g \Gamma g^{-1} = \Gamma\}$ is the normalizer of $\Gamma$ in $G$. We will define the desired sequence of functions on each piece separately.

First consider $S_1$. Enumerate $G$ as $t_0, t_1, \ldots$ and for $I \subseteq \N$ set $T_I = \{t_i : i \in I\}$. Let $(\pi_n)_{n \in \N}$ enumerate the prime numbers greater than or equal to $3$ and pick bijections $q_n : \N \rightarrow \N^{\pi_n}$ for each $n \in \N$. Now for each $\Gamma \in S_1$ and $n \in \N$ set $R_\Gamma^n = T_{q_n(\ell)}$ where $\ell$ is least so that the following conditions hold:
\begin{enumerate}
\item [\rm (a)] $R_\Gamma^n \subseteq N_G(\Gamma)$;
\item [\rm (b)] $|\Gamma \backslash R_\Gamma^n| = \pi_n$;
\item [\rm (c)] if $g \in G$ and $\Gamma g R_\Gamma^n = \Gamma R_\Gamma^n$ then $g \in \Gamma$.
\end{enumerate}
By working in the group $N_G(\Gamma) / \Gamma$, (c) requires that $R_\Gamma^n \Gamma / \Gamma$ not be a union of cosets of a cyclic subgroup. However, (b) asserts that $R_\Gamma^n \Gamma / \Gamma$ has cardinality $\pi_n$, a prime strictly greater than $2$. So (c) is equivalent to the requirement that $R_\Gamma^n \Gamma / \Gamma$ not be a cyclic subgroup. Therefore it is easy to see that for every $\Gamma$ and every $n$ there is $\ell$ so that $R_\Gamma^n = T_{q_n(\ell)}$ has the desired properties. Furthermore the function $(\Gamma, n) \mapsto R_\Gamma^n$ is Borel.

Now we check ($*$) on $S_1$. Fix $\Gamma \in S_1$, $g \in G$, and $n, k \in \N$. First suppose $g \not\in N_G(\Gamma)$. Note that $g^{-1} \Gamma g R_{g^{-1} \Gamma g}^k \subseteq N_G(g^{-1} \Gamma g)$ since $g^{-1} \Gamma g \in S_1$. Therefore
$$\Gamma g R_{g^{-1} \Gamma g}^k \subseteq g N_G(g^{-1} \Gamma g) = N_G(\Gamma) g$$
is disjoint with $N_G(\Gamma) \supseteq \Gamma R_\Gamma^n$, so $\Gamma R_\Gamma^n \neq \Gamma g R_{g^{-1} \Gamma g}^k$ as required. Now suppose $g \in N_G(\Gamma)$. Then
$$\Gamma g R_{g^{-1} \Gamma g}^k = \Gamma g R_\Gamma^k$$
If $n = k$ and $g \not\in \Gamma$ then the above set is distinct from $\Gamma R_\Gamma^n$ by (c). If $n \neq k$ then the above set is distinct from $\Gamma R_\Gamma^k$ because after quotienting by $\Gamma$ they have distinct cardinalities by (b). We conclude ($*$) holds on $S_1$.

Now consider $S_2$. Notice that the conjugation action of $G$ on $S_2$ is aperiodic. So we can pick an aperiodic Borel bijection $T : S_2 \rightarrow S_2$ such that $T(\Gamma)$ is conjugate to $\Gamma$ for every $\Gamma \in S_2$ \cite[Lem. 3.25]{JKL}. Choose any Borel function $c : \Z \times S_2 \rightarrow G$ satisfying the rule $c(m, \Gamma)^{-1} \Gamma c(m, \Gamma) = T^m(\Gamma)$ and satisfying $c(0, \Gamma) = 1_G$. Now for $\Gamma \in S_2$ and $n \in \N$ set
$$R_\Gamma^n = \{c(i, \Gamma) : 0 \leq i \leq n\}.$$

We now verify ($*$) on $S_2$. Define $q : \bigcup_{\Gamma \in S_2} \Gamma \backslash G \rightarrow S_2$ by $q(\Gamma u) = u^{-1} \Gamma u$. Fix $\Gamma \in S_2$, $g \in G$, and $n, k \in \N$. Notice that
\begin{align*}
q(\Gamma R_\Gamma^n) & = \{q(\Gamma c(i, \Gamma)) : 0 \leq i \leq n\}\\
 & = \{ c(i, \Gamma)^{-1} \Gamma c(i, \Gamma) : 0 \leq i \leq n\} = \{T^i(\Gamma) : 0 \leq i \leq n\}
\end{align*}
and
$$q(\Gamma g R_{g^{-1} \Gamma g}^k) = q(g^{-1} \Gamma g R_{g^{-1} \Gamma g}^k) = \{T^i(g^{-1} \Gamma g) : 0 \leq i \leq k\}.$$
In particular, $|q(\Gamma R_\Gamma^n)| = n + 1$ and $|q(\Gamma g R_{g^{-1} \Gamma g}^k)| = k + 1$, so we are done if $n \neq k$. Additionally, if $g \not\in N_G(\Gamma)$ then $\Gamma \neq g^{-1} \Gamma g$ and therefore the partial $T$-orbits $\{T^i(\Gamma) : 0 \leq i \leq n\}$ and $\{T^i(g^{-1} \Gamma g) : 0 \leq i \leq k\}$ have different starting points and are unequal. Finally, suppose $g \in N_G(\Gamma)$ but $g \not\in \Gamma$. By definition of $c$ we have $1_G \in R_{g^{-1} \Gamma g}^k$ and hence
$$\Gamma g R_{g^{-1} \Gamma g}^k \cap (N_G(\Gamma) \setminus \Gamma) \supseteq \Gamma g \neq \varnothing.$$
On the other hand, $c(\Z \times \{\Gamma\}) \subseteq \{1_G\} \cup (G \setminus N_G(\Gamma))$ and hence
$$\Gamma R_\Gamma^n \cap (N_G(\Gamma) \setminus \Gamma) = \varnothing.$$
We conclude that ($*$) holds on $S_2$.

Now let us complete the proof. Notice that ($*$) implies the more general statement: for all $n, k \in \N$, $\Gamma \in S$, and $g, h \in G$:
\begin{enumerate}
\item [($*'$)] if $g h^{-1} \not\in \Gamma$ or $n \neq k$ then $\Gamma g R_{g^{-1} \Gamma g}^n \neq \Gamma h R_{h^{-1} \Gamma h}^k$.
\end{enumerate}
Indeed, setting $t = h^{-1} g$ and $\Gamma' = h^{-1} \Gamma h$, we have $g h^{-1} \in \Gamma \Leftrightarrow t \in \Gamma'$,
$$\Gamma h R_{h^{-1} \Gamma h}^k = h \cdot \Gamma' R_{\Gamma'}^k \qquad \text{and} \qquad \Gamma g R_{g^{-1} \Gamma g}^n = h \cdot \Gamma' t R_{t^{-1} \Gamma' t}^n.$$
So this follows immediately from ($*$).

For $y \in Y$ set $\mu_y = \kappa^{\Gamma_y \backslash G} \times \delta_y$. Fix any function $f \in \BL^2_\kappa(K) \ominus \C$ with $\|f\|_\kappa = 1$ and define the functions $\zeta_n \in \BL^2(X, \mu) \ominus \BL^2(Y, \nu)$ by
$$\zeta_n(z, y) = \prod \{f(z(r)) : r \in R_{\Gamma_y}^n\} \qquad \qquad (z, y) \in X = K^G \times Y$$
Notice that if $R, R' \subseteq G$ are finite, $\Gamma R \neq \Gamma R'$, and $\omega, \omega' : K^G \rightarrow \C$ are defined by
$$\omega_i(z) = \prod \{f(z(r)) : r \in R\} \qquad \omega'(z) = \prod \{f(z(r)) : r \in R'\},$$
then $\langle \omega, \omega' \rangle_{\kappa^{\Gamma \backslash G}} = 0$.

Since
\begin{align*}
(\rho(u_g) \zeta_n)(z, y) = \zeta_n(g^{-1} \cdot z, g^{-1} \cdot y) & = \prod \{f((g^{-1} \cdot z)(r)) : r \in R_{\Gamma_{g^{-1} \cdot y}}^n\}\\
 & = \prod \{f(z(r)) : r \in g R_{g^{-1} \Gamma_y g}^n\},
\end{align*}
property ($*'$) implies that for every $y \in Y$, $n, k \in \N$, and $g, h \in G$, if $n \neq k$ or $g h^{-1} \not\in \Gamma_y$ then
$$\langle \rho(u_g) \zeta_n, \rho(u_h) \zeta_k \rangle_{\mu_y} = 0.$$
On the other hand, if $n = k$ and $g h^{-1} \in \Gamma_y$ then $\rho(u_g) \zeta_n$ and $\rho(u_h) \zeta_n$ are equal $\mu_y$-almost-everywhere.

Since $\BL^\infty_\nu(Y)$ acts via scalar multiplication on each $Y$-fiber, it follows from the previous paragraph that the closed $\rho(\BL^\infty_\nu(Y) \aprod G)$-invariant subspaces generated by $\zeta_n$, $n \in \N$, are pairwise-orthogonal. It suffices to show that the restriction of $\rho$ to each of these subspaces is isomorphic to $\lambda$. Define $Y_g = \{y \in Y : g \in \Gamma_y\}$. For $s, t \in \BL^\infty_\nu(Y)$ and $g, h \in G$, we use the previous paragraph to compute
\begin{align*}
\langle \rho(s u_g) \zeta_n, \rho(t u_h) \zeta_n \rangle & = \int_Y s(y) \overline{t(y)} \langle \rho(u_g) \zeta_n, \rho(u_h) \zeta_k \rangle_{\mu_y} \ d \nu(y)\\
 & = \int_{Y_{g h^{-1}}} s(y) \overline{t(y)} \ d \nu(y).
\end{align*}
This is identical to (\ref{eqn:leftreg}), thus the map $\lambda(s u_g) \xi \mapsto \rho(s u_g) \zeta_n$ extends by linearity and continuity to a $\BL^\infty_\nu(Y) \aprod G$-equivariant linear isometric embedding.
\end{proof}

\begin{cor} \label{cor:leftreg}
Let $G \acts (Y, \nu)$ be a {\pmp} action and $\omega$ be a $G$-invariant Borel probability measure on $\Sub(G) \times Y$ with $\Gamma$ having infinite index in $G$ for $\omega$-almost-every $(\Gamma, y)$ and with $\omega$ pushing forward to $\nu$ under the projection map. Let
$$\lambda : \BL^\infty_\nu(Y) \aprod G \rightarrow \cB \left( \int_{\Sub(G) \times Y}^\oplus \ell^2(\Gamma \backslash G) \ d \omega(\Gamma, y) \right)$$
be the generalized left-regular $*$-representation. Fix a non-trivial probability space $(K, \kappa)$ and consider the action $G \acts (X, \mu)$ where
$$(X, \mu) = \left(K^G \times Y, \int_{\Sub(G) \times Y} \kappa^{\Gamma \backslash G} \times \delta_y \ d \omega(\Gamma, y) \right).$$
Let $\rho : \BL^\infty_\nu(Y) \aprod G \rightarrow \cB(\BL^2_\mu(X))$ be the generalized Koopman $*$-representation. Then $\lambda^{\oplus \N}$ embeds into $\rho|_{\BL^2_\mu(X) \ominus \BL^2_\nu(Y)}$.
\end{cor}

\begin{proof}
Let $\pi : X \rightarrow Y$ be the projection map, set $Y' = \Sub(G) \times Y$, and set $\nu' = (\Stab \times \pi)_*(\mu)$. Notice that $G \acts (Y', \nu')$ is an intermediary factor between $(X, \mu)$ and $(Y, \nu)$. Also notice that $G \acts (X, \mu)$ is isomorphic to $G \acts (X', \mu')$ where
$$(X', \mu') = \left(K^G \times Y', \int_{(\Gamma, y) \in Y'} \kappa^{\Gamma \backslash G} \times \delta_{\Gamma} \times \delta_y \ d \nu'(\Gamma, y) \right)$$
(these are isomorphic since the stabilizer map on $K^G$ is $\kappa^{\Gamma \backslash G}$-almost-every equal to $\Gamma$, whenever $\Gamma$ has infinite index in $G$). For $y' = (\Gamma, y)$ set $\Gamma_{y'} = \Gamma$. Let
$$\lambda' : \BL^\infty_{\nu'}(Y') \aprod G \rightarrow \cB \left( \int_{Y'}^\oplus \ell^2(\Gamma_{y'} \backslash G) \ d \nu'(y') \right)$$
be the generalized left-regular $*$-representation, and let $\rho' : \BL^\infty_{\nu'}(Y') \aprod G \rightarrow \cB(\BL^2_{\mu'}(X'))$ be the generalized Koopman $*$-representation. Clearly if we restrict $\lambda'$ and $\rho'$ to be $*$-representations of $\BL^\infty_{\nu}(Y) \aprod G$ then $\lambda$ is isomorphic to $\lambda'$ and $\rho'|_{\BL^2_{\mu'}(X') \ominus \BL^2_{\nu'}(Y')}$ embeds into $\rho|_{\BL^2_\mu(X) \ominus \BL^2_\nu(Y)}$. Thus we are finished by the previous lemma.
\end{proof}

\begin{cor} \label{cor:sinaiembed}
Let $G \acts (X, \mu)$ be an aperiodic {\pmp} action, let $\Sigma$ be a $G$-invariant sub-$\sigma$-algebra, and let $G \acts (Y, \nu)$ be the factor associated with $\Sigma$, say via $\phi : (X, \mu) \rightarrow (Y, \nu)$. Set $\omega = (\Stab \times \phi)_*(\mu)$ and let
$$\lambda : \BL^\infty_\nu(Y) \aprod G \rightarrow \cB \left( \int_{\Sub(G) \times Y}^\oplus \ell^2(\Gamma \backslash G) \ d \omega(\Gamma, y) \right)$$
and $\rho : \BL^\infty_\nu(Y) \aprod G \rightarrow \cB(\BL^2(X, \mu))$ be the generalized left-regular and Koopman $*$-representations, respectively. If $\rh_G(X, \mu \given \Sigma) > 0$ then $\lambda^{\oplus \N}$ embeds into $\rho|_{\BL^2_\mu(X) \ominus \BL^2_\nu(Y)}$.
\end{cor}

\begin{proof}
Since $\rh_G(X, \mu \given \Sigma) > 0$, our factor theorem from the previous paper \cite{S18} (of which the previously stated Theorem \ref{intro:sinai} was just a special case) implies that there is a non-trivial probability space $(K, \kappa)$ and a factor map $f$ from $G \acts X$ to $K^G$ such that
$$(f \times \phi)_*(\mu) = \int_{\Sub(G) \times Y} \kappa^{\Gamma \backslash G} \times \delta_y \ d \omega(\Gamma, y).$$
The claim now follows from the previous corollary.
\end{proof}

\section{Spectral consequences of positive entropy}

In this section we prove our main theorem. This generalizes to Rokhlin entropy a similar result obtained by Ben Hayes for sofic entropy \cite{Ha}. However, our proof shares no similarity with the Hayes' sofic entropy proof, as his methods are entirely rooted in soficity.

Given a $*$-algebra $A$ and two $*$-representations $\pi_i : A \rightarrow \cB(\HS_i)$, $i = 1, 2$, we say that $\pi_1$ and $\pi_2$ are \emph{singular} if no non-trivial subrepresentation of $\pi_1$ is (isometrically) embeddable into $\pi_2$. Also, for a sub-$\sigma$-algebra $\cF$ of $(X, \mu)$, we write $\BL^2_\mu(\cF)$ for the set of $\cF$-measurable functions in $\BL^2_\mu(X)$, and conversely for $\HS \subseteq \BL^2_\mu(X)$ we write $\salg_G(\HS)$ for the smallest $G$-invariant $\sigma$-algebra containing $\psi^{-1}(D)$ for $\psi \in \HS$, $D \subseteq \C$ Borel.

\begin{thm} \label{thm:koop}
Let $G \acts (X, \mu)$ be an aperiodic {\pmp} action, and let $G \acts (Y, \nu)$ be a factor action, say via $\phi : (X, \mu) \rightarrow (Y, \nu)$. Set $\omega = (\Stab \times \phi)_*(\mu)$ and $\Sigma = \phi^{-1}(\Borel(Y))$. Let $\rho : \BL^\infty_\nu(Y) \aprod G \rightarrow \BL^2_\mu(X)$ and
$$\lambda : \BL^\infty_\nu(Y) \aprod G \rightarrow \cB \left( \int_{\Sub(G) \times Y}^{\oplus} \ell^2(\Gamma \backslash G) \ d \omega(\Gamma, y) \right)$$
be the generalized Koopman and left-regular representations, respectively. If $\HS$ is a $\rho(\BL^\infty_\nu(Y) \aprod G)$-invariant closed subspace of $\BL^2_\mu(X)$ and $\rho|_{\HS}$ is singular with respect to $\lambda$ then $\rh_G(\salg_G(\HS) \given \Sigma) = 0$.
\end{thm}

The proof of this theorem will rely upon some of the methods developed in our previous paper \cite{S18}, which we now review. Recall that for a {\pmp} action $G \acts (X, \mu)$ with associated orbit-equivalence relation $E_G^X = \{(x, y) : \exists g \in G \ g \cdot x = y\}$, the associated \emph{full group}, denoted $[E_G^X]$, is the group of all Borel bijections $T : X \rightarrow X$ (identified up to equality $\mu$-almost-everywhere) satisfying $T(x) \ E_G^X \ x$ for $\mu$-almost-every $x \in X$. If $T \in [E_G^X]$ and $\cF$ is a $G$-invariant sub-$\sigma$-algebra, we say that $T$ is \emph{$\cF$-expressible} if there is a $\cF$-measurable partition of $X$, $\{Z_g : g \in G\}$, satisfying $T(x) = g \cdot x$ for every $g \in G$ and $\mu$-almost-every $x \in Z_g$.

Recall that for a {\pmp} action $G \acts (X, \mu)$ and a $G$-invariant sub-$\sigma$-algebra $\cF$, there exists a unique (up to isomorphism) factor action $G \acts (Y, \nu)$, say via $\phi : (X, \mu) \rightarrow (Y, \nu)$, satisfying $\phi^{-1}(\Borel(Y)) = \cF$. The sub-$\sigma$-algebra $\cF$ is called \emph{class-bijective} if $\Stab(\phi(x)) = \Stab(x)$ for $\mu$-almost-every $x \in X$.

\begin{lem}[Seward, \cite{S18}] \label{lem:slice}
Let $G \acts (X, \mu)$ be an aperiodic {\pmp} action, and let $\cF$ be a $G$-invariant class-bijective sub-$\sigma$-algebra. Then there is an aperiodic $\cF$-expressible $T \in [E_G^X]$ and a $T$-invariant $\cF$-measurable function $f : X \rightarrow [0, 1]$ with the property that $x, y \in X$ lie in the same $G$-orbit and have equal $f$ values if and only if they lie in the same $T$-orbit.
\end{lem}

An important consequence of the above lemma is that it creates an ordering on each $G$-orbit, and with it a useful notion of ``past.'' Specifically, for $x \in X$ we define a quasi-order on $G$ by setting $u \leq_x v$ if either $f(u^{-1} \cdot x) < f(v^{-1} \cdot x)$ or $f(u^{-1} \cdot x) = f(v^{-1} \cdot x)$ and there is $m \geq 0$ with $T^{-m}(u^{-1} \cdot x) = v^{-1} \cdot x$. When $x$ has trivial stabilizer $\leq_x$ is a total order on $G$. We write $u <_x v$ when $u \leq_x v$ but $\neg (v \leq_x u)$. Its a simple consequence of the $\cF$-expressibility of $T$ and the $\cF$-measurability of $f$ that for every $u, v \in G$ the set $\{x \in X : u <_x v\}$ is $\cF$-measurable. The past of $x$ consists of the points $\{g^{-1} \cdot x : g <_x 1_G\}$, and this past breaks into two pieces -- a part that is internal to the $T$-action and a part that is external. This concept leads to the notion of a past $\sigma$-algebra for any given partition. We next describe in detail this $\sigma$-algebra for the external portion of the past.

For a partition $\alpha$ of $X$ and $Y \subseteq X$ we define another partition of $X$ by
$$\alpha \res Y = \{X \setminus Y\} \cup \{A \cap Y : A \in \alpha\}.$$
Similarly, for a $\sigma$-algebra $\Sigma$ we write $\Sigma \res Y$ for the $\sigma$-algebra on $X$ generated by the sets $\{X\} \cup \{B \cap Y : B \in \Sigma\}$.

\begin{defn}
Let $G \acts (X, \mu)$ be an aperiodic {\pmp} action, let $\cF$ be a $G$-invariant class-bijective sub-$\sigma$-algebra, and let $T \in [E_G^X]$ and $f : X \rightarrow [0, 1]$ be as in Lemma \ref{lem:slice}. For a countable partition $\xi$ of $X$ and $S \subseteq [0, 1]$ write $\xi_S$ for the partition $\xi_S = \xi \res f^{-1}(S)$. We define the \emph{external past} of $\xi$ as
$$\sP_\xi = \cF \vee \bigvee_{t \in [0, 1]} \Big( \salg_G(\xi_{[0,t)}) \res f^{-1}([t,1]) \Big).$$
\end{defn}

In other words, the external past of $\xi$ consists of the sets that you can measure by using $G$ to travel to strictly ``smaller'' $T$-orbits and looking at the partition $\xi$ (we also include $\cF$ in this $\sigma$-algebra for technical reasons). The next lemma verifies that these $\sigma$-algebras behave as expected with respect to containment.

\begin{lem}[Seward, \cite{S18}] \label{lem:order}
Let $G \acts (X, \mu)$ be an aperiodic {\pmp} action, let $\cF$ be a $G$-invariant class-bijective sub-$\sigma$-algebra, and let $T \in [E_G^X]$ and $f : X \rightarrow [0, 1]$ be as in Lemma \ref{lem:slice}. Fix a countable partition $\xi$ of $X$. For all $u, v \in G$ and every $D \in \cF$ with $D \subseteq \{x \in X : u <_x v\}$ we have
$$\Big[ u \cdot \Big( \textstyle{\bigvee_{k \leq 0} T^k(\xi)} \vee \sP_\xi \Big) \Big] \res D \subseteq v \cdot \Big( \textstyle{\bigvee_{k < 0} T^k(\xi)} \vee \sP_\xi \Big).$$
\end{lem}

Finally, an important feature of the external past is how it relates Rokhlin entropy to classical Kolmogorov--{\sinai} entropy.

\begin{lem}[Seward, \cite{S18}] \label{lem:sent}
Let $G \acts (X, \mu)$ be an aperiodic {\pmp} action, let $\cF$ be a $G$-invariant class-bijective sub-$\sigma$-algebra, and let $T \in [E_G^X]$ and $f : X \rightarrow [0, 1]$ be as in Lemma \ref{lem:slice}. If $\xi$ is a countable partition with $\sH(\xi \given \cF) < \infty$, then
$$\rh_G(\xi \given \cF) \leq \ksh_T(\xi \given \sP_\xi),$$
where $\ksh$ denotes Kolmogorov--{\sinai} entropy.
\end{lem}

Before presenting the proof of the main theorem we need one more simple lemma. In the remainder of this section, we write $\Exp_{\cF}$ for the projection operator on $\BL^2_\mu(X)$ given by taking the conditional expectation with respect to a sub-$\sigma$-algebra $\cF$.

\begin{lem} \label{lem:convex}
Let $G \acts (X, \mu)$ be a {\pmp} action, and let $\cF$ be a $G$-invariant sub-$\sigma$-algebra. If $\HS$ satisfies $\rh_G(\salg_G(\HS) \given \cF) > 0$ then there is a finite partition $\alpha$ with $\rh_G(\alpha \given \cF) > 0$ and satisfying for all sub-$\sigma$-algebras $\cA$
$$\alpha \subseteq \cA \Longleftrightarrow \alpha \subseteq \salg(\Exp_{\cA}(\HS)).$$
\end{lem}

\begin{proof}
Clearly $\salg(\Exp_{\cA}(\HS)) \subseteq \cA$, so the right-hand side always implies the left. For the reverse implication, it suffices to pick any partition $\beta$ of $\C$ into convex sets (for example, dividing the complex plane using a finite number of bi-infinite straight lines) and pick any $\psi \in \HS$ and set $\alpha = \psi^{-1}(\beta)$. Indeed, assuming $\alpha \subseteq \cA$, its a basic property of conditional expectation that if $A \in \alpha$ and $x \in A$ then $\Exp_{\cA}(\psi)(x)$ will lie in the (not necessarily closed) convex set generated by $\psi(A)$. By construction, the sets $\psi(A)$, $A \in \alpha$, are contained in the pairwise-disjoint convex sets $B \in \beta$. Therefore $\alpha = \Exp_{\cA}(\psi)^{-1}(\beta)$, and hence $\alpha \subseteq \salg(\Exp_{\cA}(\HS))$. To finish the proof, we only need to find a finite $\alpha$ with $\rh_G(\alpha \given \cF) > 0$.

Let $(\psi_n)_{n \in \N}$ be a dense subset of $\HS$. For any sub-$\sigma$-algebra $\Sigma$, $\BL^2$-limits of $\Sigma$-measurable functions are $\Sigma$-measurable. So we have $\salg(\HS) = \bigvee_{n \in \N} \salg(\psi_n)$. Sub-additivity of Rokhlin entropy \cite{AS} gives
$$0 < \rh_G(\salg(\HS) \given \cF) \leq \sum_{n \in \N} \rh_G(\salg(\psi_n) \given \cF).$$
So there is $\psi \in \HS$ with $\rh_G(\salg(\psi) \given \cF) > 0$. Let $(\beta_n)_{n \in \N}$ be an increasing sequence of finite partitions of $\C$ with $\Borel(\C) = \bigvee_{n \in \N} \salg(\beta_n)$ and with all classes of $\beta_n$ convex. Then $\salg(\psi) = \bigvee_{n \in \N} \salg(\psi^{-1}(\beta_n))$. Again by sub-additivity $0 < \rh_G(\salg(\psi) \given \cF) \leq \sum_{n \in \N} \rh_G(\psi^{-1}(\beta_n) \given \cF)$. We conclude that there is a finite partition $\beta$ of $\C$ with all classes of $\beta$ convex, and $\psi \in \HS$ such that $\alpha = \psi^{-1}(\beta)$ satisfies $\rh_G(\alpha \given \cF) > 0$.
\end{proof}

\begin{proof}[Proof of Theorem \ref{thm:koop}]
We will prove the contra-positive. So let us assume that $\rh_G(\salg(\HS) \given \Sigma) > 0$. Let $\alpha$ be as in Lemma \ref{lem:convex}. Then $\rh_G(\alpha \given \Sigma) > 0$. By \cite[Lem. 2.3]{S18} we can enlarge $\Sigma$ to a $G$-invariant $\sigma$-algebra $\cF$ so that $\cF$ is class-bijective and $\rh_G(\alpha \given \cF) > 0$. Note that the map $\Stab$ is $\cF$-measurable since $\cF$ is class-bijective. Apply Lemma \ref{lem:slice} to get a $\cF$-expressible $T \in [E_G^X]$ and a $\cF$-measurable $f : X \rightarrow [0, 1]$.

Set $\cA = \bigvee_{n \leq 0} \salg(T^n(\alpha)) \vee \sP_\alpha$ and $\cA^- = \bigvee_{n < 0} \salg(T^n(\alpha)) \vee \sP_\alpha$. By Lemma \ref{lem:sent} we have
$$0 < \rh_G(\alpha \given \cF) \leq \ksh_T(\alpha \given \sP_\alpha) = \sH(\alpha \given \cA^-).$$
So $\alpha \not\subseteq \cA^-$ but $\alpha \subseteq \cA$. It follows from Lemma \ref{lem:convex} that
$$\HS_{1_G} = (\Exp_{\cA} - \Exp_{\cA^-})(\HS) \subseteq \BL^2_\mu(\cA) \ominus \BL^2_\mu(\cA^-)$$
is non-trivial.

Fix any non-trivial $\zeta \in \HS_{1_G}$. Let $\mu = \int_{\Sub(G) \times Y} \mu_{(\Gamma, y)} \ d \omega(\Gamma, y)$ be the disintegration of $\mu$ over $\omega$. For $g, h \in G$ we claim that $\rho(u_g) \zeta$ and $\rho(u_h) \zeta$ are $\mu_{(\Gamma, y)}$-orthogonal if $g h^{-1} \not\in \Gamma$ and that $\rho(u_g) \zeta$ and $\rho(u_h) \zeta$ agree $\mu_{(\Gamma, y)}$-almost-everywhere otherwise. The latter case is immediate as $g^{-1} \cdot x = h^{-1} \cdot x$ for $\mu_{(\Gamma, y)}$-almost-every $x$ whenever $g h^{-1} \in \Gamma$. Now consider the former case. Partition $\{x \in X : g h^{-1} \not\in \Stab(x)\}$ into the sets $D_1 = \{x \in X : g <_x h\}$ and $D_2 = \{x \in X : h <_x g\}$, and write $\chi_i$ for the characteristic function of $D_i$. Recalling that $\zeta$ is $\mu$-orthogonal to $\BL^2_\mu(\cA^-)$ and that $\chi_i$ is $\cF \subseteq \cA^-$-measurable, it follows that $\zeta \chi_i$ is $\mu$-orthogonal to $\BL^2_\mu(\cA^-) \chi_i$. Additionally, the map $\Stab \times \phi$ is measurable with respect to $\cF \subseteq \cA^-$, so $\zeta \chi_i$ is $\mu_{(\Gamma, y)}$-orthogonal to $\BL^2_{\mu_{(\Gamma, y)}}(\cA^-) \chi_i$ for $\omega$-almost-every $(\Gamma, y)$. Therefore by employing Lemma \ref{lem:order} we obtain
$$(\rho(u_g) \zeta) \cdot \chi_1 \in \BL^2_{\mu_{(\Gamma, y)}}(g \cdot \cA) \cdot \chi_1 \subseteq \BL^2_{\mu_{(\Gamma, y)}}(h \cdot \cA^-) \cdot \chi_1 \perp_{\mu_{(\Gamma, y)}} \rho(u_h) \zeta \cdot \chi_1$$
$$(\rho(u_h) \zeta) \cdot \chi_2 \in \BL^2_{\mu_{(\Gamma, y)}}(h \cdot \cA) \cdot \chi_2 \subseteq \BL^2_{\mu_{(\Gamma, y)}}(g \cdot \cA^-) \cdot \chi_2 \perp_{\mu_{(\Gamma, y)}} \rho(u_g) \zeta \cdot \chi_2.$$
Finally, when $g h^{-1} \not\in \Gamma$ we have $\chi_1 + \chi_2 = 1$ $\mu_{(\Gamma, y)}$-almost-everywhere and thus $\langle \rho(u_g) \zeta, \rho(u_h) \zeta \rangle_{\mu_{(\Gamma, y)}} = 0$ as claimed.

For $(\Gamma, y) \in \Sub(G) \times Y$ write $\|\zeta\|_{(\Gamma, y)}$ for $\|\zeta\|_{\mu_{(\Gamma, y)}}$ and define a vector $\eta \in \int_{\Sub(G) \times Y}^\oplus \ell^2(\Gamma \backslash G) \ d \omega(\Gamma, y)$ by
$$\eta(\Gamma, y) = \|\zeta\|_{(\Gamma, y)} \cdot \delta_{\Gamma}.$$
Set $S_g = \{\Gamma \in \Sub(G) : g \in \Gamma\}$. For $s, t \in \BL^\infty_\nu(Y)$ and $g, h \in G$ we have
\begin{align*}
\langle \rho(s u_g) \zeta, \rho(t u_h) \zeta \rangle_\mu &= \int_{\Sub(G) \times Y} s(y) \overline{t(y)} \langle \rho(u_g) \zeta, \rho(u_h) \zeta \rangle_{\mu_{(\Gamma, y)}} \ d \omega(\Gamma, y)\\
 & = \int_{S_{g h^{-1}} \times Y} s(y) \overline{t(y)} \cdot \|\rho(u_g) \zeta\|^2_{\mu_{(\Gamma, y)}} \ d \omega(\Gamma, y)\\
 & = \int_{S_{g h^{-1}} \times Y} s(y) \overline{t(y)} \|\zeta\|_{g^{-1} \cdot (\Gamma, y)}^2 \ d \omega(\Gamma, y)\\
 & = \int_{\Sub(G) \times Y} \left\langle s(y) \|\zeta\|_{g^{-1} (\Gamma, y)} \delta_{\Gamma g}, \ t(y) \|\zeta\|_{h^{-1} (\Gamma, y)} \delta_{\Gamma h} \right\rangle \ d \omega(\Gamma, y)\\
 & = \langle \lambda(s u_g) \eta, \lambda(t u_h) \eta \rangle.
\end{align*}
Thus the map $L : \rho(s u_g) \zeta \mapsto \lambda(s u_g) \eta$ is a $\BL^\infty_\nu(Y) \aprod G$-equivariant linear isometry. By continuity $L$ extends to $V$, the closed $\rho(\BL^\infty_\nu(Y) \aprod G)$-invariant linear subspace containing $\zeta$. Finally, by construction the projection from $\HS$ to $V$, $p_V(\HS)$, is non-trivial. The composition $L \circ p_V$ is a $\BL^\infty_\nu(Y) \aprod G$-equivariant bounded linear map from $\HS$ into $\int_{\Sub(G) \times Y}^\oplus \ell^2(\Gamma \backslash G) \ d \omega(\Gamma, y)$. From the Polar Decomposition Theorem it follows that there is a closed $\rho(\BL^\infty_\nu(Y) \aprod G)$-invariant subspace of $\HS$ that isometrically and $\BL^\infty_\nu(Y) \aprod G$-equivariantly embeds into $\int_{\Sub(G) \times Y}^\oplus \ell^2(\Gamma \backslash G) \ d \omega(\Gamma, y)$.
\end{proof}

As an immediate corollary we are able to compute the Rokhlin entropy of some Gaussian actions. For a real Hilbert space $\HS$ we write $\mathcal{O}(\HS)$ for the set of orthogonal operators on $\HS$. Recall that from any representation $\pi : G \rightarrow \mathcal{O}(\HS)$ one can build an associated {\pmp} action $G \acts (X_\pi, \mu_\pi)$ called the \emph{Gaussian action} induced by $\pi$ (see \cite{Hb} for details).

\begin{cor} \label{cor:gaussian}
Let $G$ be a countably infinite group, let $\pi : G \rightarrow \mathcal{O}(\HS)$ be an orthogonal representation on a real separable Hilbert space $\HS$, and let $\lambda_\R : G \rightarrow \mathcal{O}(\ell^2(G, \R))$ be the real left-regular representation of $G$. Suppose that $\pi$ is singular with $\lambda_\R$. Then the corresponding Gaussian action $G \acts (X_\pi, \mu_\pi)$ satisfies $\rh_G(X_\pi, \mu_\pi) = 0$.
\end{cor}

\begin{proof}
We sketch the proof as the details of the argument have already been worked out by Hayes in \cite{Hb}. As argued by Hayes, one can complexify $\pi$ to get a unitary representation $\pi_\C : G \rightarrow \cU(\HS_\C)$, where $\HS_\C = \HS \otimes_\R \C$. This new representation $\pi_\C$ will be singular with $\lambda$, where $\lambda$ is the usual (complex) left-regular representation. Furthermore, the Koopman representation $\rho : G \rightarrow \cU(\BL^2_{\mu_\pi}(X_\pi))$ admits a closed sub-representation $\mathcal{K}$ such that $\rho|_{\mathcal{K}}$ is isomorphic to $\pi_\C$ and such that $\salg_G(\mathcal{K}) = \Borel(X_\pi)$. Thus it follows from Theorem \ref{thm:koop} that
\begin{equation*}
\rh_G(X_\pi, \mu_\pi) = \rh_G(\salg_G(\mathcal{K})) = 0.\qedhere
\end{equation*}
\end{proof}

\begin{cor} \label{cor:equal}
Let $G$ be a countably infinite sofic group and let $\pi : G \rightarrow \mathcal{O}(\HS)$ be an orthogonal representation on a real separable Hilbert space $\HS$. Then the Rokhlin entropy of the corresponding Gaussian action $G \acts (X_\pi, \mu_\pi)$ is equal to the sofic entropy whenever the sofic entropy is not minus infinity.
\end{cor}

\begin{proof}
When combined with work of Hayes \cite{Hb}, this is an immediate consequence since sofic entropy is always both (a) bounded above by Rokhlin entropy, and (b) takes values in $\{-\infty\} \cup [0, +\infty]$. Specifically, in \cite{Hb} Hayes computed the sofic entropy of Gaussian actions under the assumption that their sofic entropy is not minus infinity. He showed that the sofic entropy is $0$ when the assumptions of Corollary \ref{cor:gaussian} hold and otherwise the sofic entropy is positive infinity.
\end{proof}

\section{Completely positive entropy and mixing}

Putting all of our results together, we are able to fully determine the $\BL^2$ structure of CPE$^+$ (and many CPE) actions. Furthermore, our results can be applied in a relative setting as well. Let $G \acts (X, \mu)$ be a {\pmp} action and let $G \acts (Y, \nu)$ be a factor action, say via $\phi : (X, \mu) \rightarrow (Y, \nu)$. We say that $(X, \mu)$ is \emph{CPE relative to} $(Y, \nu)$ if for every non-trivial intermediary factor
$$G \acts (X, \mu) \xrightarrow{f} G \acts (Z, \eta) \xrightarrow{\psi} G \acts (Y, \nu) \quad (\psi \circ f = \phi)$$
we have $\rh_G(Z, \eta \given \psi^{-1}(\Borel(Y))) > 0$. Similarly, we say that $(X, \mu)$ is \emph{CPE$^+$ relative to} $(Y, \nu)$ is $\rh_G(\cF \given \phi^{-1}(\Borel(Y))) > 0$ whenever $\cF$ is a sub-$\sigma$-algebra that properly contains $\phi^{-1}(\Borel(Y))$.

We would like to emphasize that the corollary below is a structure theorem for all positive entropy actions. Specifically, every {\pmp} action $G \acts (X, \mu)$ admits a smallest factor $G \acts (Y, \nu)$ (and a smallest factor $G \acts (Y_+, \nu_+)$) such that $(X, \mu)$ is CPE relative to $(Y, \nu)$ (respectively CPE$^+$ relative to $(Y_+, \nu_+)$). These are called the (Rokhlin) Pinsker factor and outer Pinsker factor, respectively. When $\rh_G(X, \mu) > 0$ these factors are proper and the structure theorem of the corollary below applies.

\begin{cor} \label{cor:cpekoop}
Let $G \acts (X, \mu)$ be an aperiodic {\pmp} action, and let $G \acts (Y, \nu)$ be a factor action, say via $\phi : (X, \mu) \rightarrow (Y, \nu)$. Set $\omega = (\Stab \times \phi)_*(\mu)$. Assume that either
\begin{enumerate}
\item [\rm (1)] $G \acts (X, \mu)$ is CPE$^+$ relative to $Y$, or
\item [\rm (2)] $G \acts (X, \mu)$ is CPE relative to $Y$ and every non-trivial intermediate factor is class-bijective.
\end{enumerate}
Let
$$\lambda : \BL^\infty_\nu(Y) \aprod G \rightarrow \cB \left( \int_{\Stab \times Y}^\oplus \ell^2(\Gamma \backslash G) \ d \omega(\Gamma, y) \right)$$
and $\rho : \BL^\infty_\nu(Y) \aprod G \rightarrow \cB(\BL^2(X, \mu))$ be the generalized left-regular and Koopman representations, respectively. Then $\rho|_{\BL^2_\mu(X) \ominus \BL^2_\nu(Y)}$ is isomorphic to $\lambda^{\oplus \N}$.
\end{cor}

\begin{proof}
Set $\Sigma = \phi^{-1}(\Borel(Y))$. Clearly $\rh_G(X, \mu \given \Sigma) > 0$ and therefore $\lambda^{\oplus \N}$ embeds into $\rho|_{\BL^2_\mu(X) \ominus \BL^2_\nu(Y)}$ by Corollary \ref{cor:sinaiembed}. By Lemma \ref{lem:repiso}, it suffices to show that $\rho|_{\BL^2_\mu(X) \ominus \BL^2_\nu(Y)}$ embeds into $\lambda^{\oplus \N}$. By a simple application of Zorn's lemma, there is a maximal $\rho(\BL^\infty_\nu(Y) \aprod G)$-invariant closed subspace $\mathcal{K}$ of $\BL^2_\mu(X) \ominus \BL^2_\nu(Y)$ such that $\rho|_{\mathcal{K}}$ embeds into $\lambda^{\oplus \N}$. Set $\HS = \BL^2_\mu(X) \ominus (\mathcal{K} + \BL^2_\nu(Y))$. To complete the proof we will argue that $\HS = \{0\}$ under assumption (1) and under assumption (2).

(1). By maximality of $\mathcal{K}$, $\rho|_{\HS}$ must be singular with $\lambda$. So by Theorem \ref{thm:koop} $\rh_G(\salg_G(\HS) \given \Sigma) = 0$. The assumption of CPE$^+$ then implies that $\salg_G(\HS) \subseteq \Sigma$. But $\HS$ is orthogonal to $\BL^2_\nu(Y) = \BL^2_\mu(\Sigma)$, so we must have $\HS = \{0\}$.

(2). Towards a contradiction, suppose that $\HS \neq \{0\}$. Let $G \acts (Z, \eta)$ be the intermediary factor associated with $\salg_G(\HS) \vee \Sigma$, and let $f : (X, \mu) \rightarrow (Z, \eta)$ be the factor map. Choose $\phi' : (Z, \eta) \rightarrow (Y, \nu)$ so that $\phi' \circ f = \phi$ and set $\Sigma' = (\phi')^{-1}(\Borel(Y))$. Every function in $\HS$ is $f^{-1}(\Borel(Z))$-measurable, so we may identify $\HS$ with a subspace $\HS'$ of $\BL^2_\eta(Z)$ satisfying $\salg_G(\HS') = \Borel(Z)$. Our assumptions imply that $f$ is class-bijective and
$$\rh_G(\salg(\HS') \given \Sigma') = \rh_G(Z, \eta \given \Sigma') > 0.$$
Since $f$ is class-bijective, $(\Stab \times \phi')_*(\eta)$ coincides with $\omega$. So Theorem \ref{thm:koop} implies that $\lambda$ is not singular with $\rho|_{\HS}$, contradicting the maximality of $\mathcal{K}$.
\end{proof}

Before stating the final corollary, let us review some definitions. Let $G \acts (X, \mu)$ be a {\pmp} action, let $\Sigma$ be a $G$-invariant sub-$\sigma$-algebra, let $G \acts (Y, \nu)$ be the factor associated with $\Sigma$, and let $\mu = \int_Y \mu_y \ d \nu(y)$ be the disintegration of $\mu$ over $\nu$. Recall that $G \acts (X, \mu)$ is \emph{mixing relative to $Y$} if for all $A, B \in \Borel(X)$ we have
$$\lim_{g \rightarrow \infty} \int_Y (\mu_y(A \cap g \cdot B) - \mu_y(A) \mu_y(g \cdot B))^2 \ d \nu(y) = 0.$$
Additionally, $G \acts (X, \mu)$ has \emph{spectral gap relative to $Y$} if for every sequence $\xi_n \in \BL^2_\mu(X)$ satisfying $\forall g \in G \ \lim_{n \rightarrow \infty} \| \rho(g)(\xi_n) - \xi_n\|_\mu = 0$, we have that $\lim_{n \rightarrow \infty} \| \xi_n - \Exp_\Sigma(\xi_n) \|_\mu = 0$. Finally, $G \acts (X, \mu)$ is \emph{strongly ergodic relative to $Y$} if for every sequence of sets $A_n \in \Borel(X)$ satisfying $\forall g \in G \ \lim_{n \rightarrow \infty} \mu(A_n \symd g \cdot A_n) = 0$, we have that $\lim_{n \rightarrow \infty} \int_Y \mu_y(A_n)(1 - \mu_y(A_n)) \ d \nu(y) = 0$.

\begin{cor} \label{cor:mix}
Let $G \acts (X, \mu)$ be a free {\pmp} action and let $G \acts (Y, \nu)$ be a factor action. Assume that either
\begin{enumerate}
\item [\rm (1)] $G \acts (X, \mu)$ is CPE$^+$ relative to $Y$, or
\item [\rm (2)] $G \acts (X, \mu)$ is CPE relative to $Y$ and every non-trivial intermediary factor action is free.
\end{enumerate}
Then $G \acts (X, \mu)$ is mixing relative to $Y$, and for every non-amenable $\Gamma \leq G$ the restricted action $\Gamma \acts (X, \mu)$ has spectral gap relative to $Y$ and is strongly ergodic relative to $Y$.
\end{cor}

\begin{proof}
This is a direct consequence of Corollary \ref{cor:cpekoop}. We refer the reader to \cite{Ha} for details, but we briefly explain the idea here. Relative mixing is equivalent to a relative mixing condition on the Koopman representation. So both relative mixing and spectral gap are spectral properties. These properties hold for the generalized left-regular representation when $\omega$ is supported on $\{\{1_G\}\} \times Y$. Thus, by Corollary \ref{cor:cpekoop}, they hold for our action. Additionally, relative strong ergodicity is an immediate consequence of relative spectral gap.
\end{proof}

%\bibitem{Th75}
%J.-P. Thouvenot,
%\textit{Quelques proprietes des systemes dynamiques qui se decomposent en un produit de deux systemes dont l'un est un schema de Bernoulli}, Israel Journal of Mathematics 21 (1975), no. 2-3, 177--206.

%\bibitem{Var63}
%V. S. Varadarajan,
%\textit{Groups of automorphisms of Borel spaces}, Trans. Amer. Math. Soc. 109 (1963), 191--220.

\end{document}